\documentclass{article}
\usepackage{amsfonts}
\usepackage{amsmath,amssymb,amsthm}
\newtheorem{definition}{Definition}
\newtheorem{theorem}{Theorem}
\newtheorem{lemma}{Lemma}
\newtheorem{corollary}{Corollary}
\begin{document}

\title{A Probabilistic Characterization of g-Harmonic Functions}
\author{Liang Cai \\
        Department of Mathematics, Beijing Institute of Technology \\ Beijing
        100081, People's Republic of China\\
        E-mail: cailiang@bit.edu.cn}
\maketitle

\begin{abstract}
This paper gives a definition of g-harmonic functions and shows the
relation between the g-harmonic functions and g-martingales. It's
direct to construct such relation under smooth case, but for
continuous case we need the theory of viscosity solution. The
results show that under the nonlinear expectation mechanism, we also
can get the similar relation between harmonic functions and
martingales. Finally, we will give a result about the strict
converse problem of mean value property of g-harmonic functions.

{\bf{Key words or phrases}}: BSDE, g-martingale, g-harmonic
function, nonlinear Feynman-Kac formula, viscosity solution.
\end{abstract}

\section{Introduction and Preliminary}
Harmonic function ($\Delta u=0$)  has a probabilistic interpretation
as that if $\Delta u=0$ on $R^n$, then $u(B^x_t)$ is a martingale
for any $x\in R^n$ (see for example \cite{oksendal1}). This relation
between martingale and harmonic function connects probability with
potential analysis. It helps us to give probabilistic
characterizations for harmonic function and more generalized
X-harmonic function(see \cite{oksendal1}). In 1997,
Peng(\cite{Peng1}) introduced the notions of g-expectation and
conditional g-expectation via a backward stochastic differential
equation(BSDE) with a generator function $g$. Further,
Peng(\cite{Peng2}) introduced the notion of g-martingale which
provided an heuristic tool to characterize some kind of harmonic
function described by elliptic operator with a nonlinear term $g$ by
which we will define g-harmonic function. In this paper, with the
help of nonlinear Feynman-Kac formula established from BSDE(see for
example \cite{Peng3}), we will give a probabilistic characterization
of the g-harmonic function.

Now we state our problem in detail. Let $(\Omega,\mathcal{F},P)$ be
a probability space endowed with the natural filtration
$\{\mathcal{F}_t\}_{t\geq0}$ generated by an n-dimensional Brownian
motion $\{B_t^x\}_{t\geq0}$. i.e.
$$\mathcal{F}_t=\sigma\{B_s: s\leq t\}.$$
Then we can define a g-martingale by an $\mathcal{F}_t$-adapted
process $\{y_t\}_{t\geq0}$ which satisfies the following BSDE for
any $0\leq s\leq t$:
\begin{align}\label{BSDE}
y_s=y_t+\int^t_sg(y_r,z_r)dr-\int^t_sz_rdB_r.
\end{align}
Here $g:R\times R^n\longrightarrow R$, satisfies the condition:\\
(H1). $g(y,0)\equiv 0$ and the Lipschitz condition: $\exists C>0$,
for any $(y_1,z_1),(y_2,z_2)\in R \times R^n$ we have
\begin{align}
|g(y_1,z_1)-g(y_2,z_2)|\leq C(|y_1-y_2|+|z_1-z_2|).\notag
\end{align}
And the equality(\ref{BSDE}) also can be formulated simply
as\cite{Peng4}:
\begin{align}
\mathcal{E}^g_{s,t}(y_t):
=y_s.\notag
\end{align}
Then we can also get the definition of g-super(sub)martingale when
\begin{align}
\mathcal{E}^g_{s,t}(y_t)\leq (\geq)\  y_s.\notag
\end{align}
This definition derives from the definition of g-expectation in the
beginning paper Peng\cite{Peng1}. When $g(y,z)\equiv 0$ the
g-expectation is actually the classical expectation. Except that
g-expectation is nonlinear in general, it holds many other important
properties as its classical counterpart
\cite{ChenPeng1}\cite{Jiang}\cite{Peng2}\cite{Wangwei}.

Given an n-dimensional It\^{o} diffusion process$\{X^x_t\}_{t\geq
0}$:
\begin{align}\label{diffusion}
dX^x_t&=b(X^x_t)dt+\sigma(X^x_t)dB_t,  \\
 X^x_0&=x \in R^n \notag ,
\end{align}
 where $\  b(x): R^n\longrightarrow R^n,\  \sigma(x): R^n\longrightarrow
R^{n\times n}$ satisfy the Lipschitz condition: $\exists C>0$ s.t.
\begin{align}
|b(x_1)-b(x_2)|+|\sigma(x_1)-\sigma(x_2)|\leq C|x_1-x_2|\ \ \
\forall x_1, x_2\in R^n,\notag
\end{align}
our problem is that: what kind of function $u(x): R^n\longrightarrow
R$ satisfies that $u(X^x_t)$ is a g-martingale for any $x\in R^n$?

This problem also has its classical counterpart:

First if $\{X^x_t\}$ is just the Brownian motion $\{B^x_t\}$, then
we have the result that when $u(x)$ is harmonic on $R^n$ i.e.
\begin{align}
\Delta u=\Sigma_i\frac{\partial^2 u}{\partial x_i^2}=0\ \ \ \ \mbox
{for any}\  x\in R^n, \notag
\end{align}
the process $u(B^x_t)$ is a martingale for any x. And conversely if
$u(x)$ satisfies that $u(B_t^x)$ is a martingale for any x, then
$u(x)$ must be harmonic on $R^n$. The proof may have many editions,
here we can give a sketch of one which may induce the extension to
g-martingale case.

If $u(x)$ is harmonic on $R^n$, then we use It\^{o} formula to
$u(B^x_t)$ and get
\begin{align}
du(B^x_t)&=\sum_i\frac{\partial u}{\partial
x_i}(B_t^x)dB_{i,t}+\frac{1}{2}\sum_i\frac{\partial^2 u}{\partial
x_i^2}(B^x_t)dt\notag\\
&=\sum_i\frac{\partial u}{\partial x_i}(B_t^x)dB_{i,t}.\notag
\end{align}
Then we get $u(B^x_t)$ is a martingale for any $x\in R^n$.
Conversely if $u(x)$ is continuous on $R^n$ and  for any $x\in R^n$,
$u(B^x_t)$ is a martingale, then we have $E[u(B^x_\tau)]=u(x)$ for
any stopping time $\tau$. Particularly for any sphere $S(x,r)=\{y\in
R^n: |y-x|< r\}$, we have
\begin{align}
u(x)=E[u(B^x_{\tau_{S(x,r)}})]=\int_{\partial S(x,r)}u(y)d\sigma_y\
\ \ \  \notag
\end{align}
where $\tau_{S(x,r)}$ is the exit time of $\{B^x_t\}$ from the
sphere $S(x,r)$, i.e.
$$\tau_{S(x,r)}=inf\{t>0: |B^x_t-x|\geq r\} ,$$
 and $\sigma_y$ is the harmonic measure on
the $\partial S(x,r) $. Then from the familiar converse of the mean
value property for harmonic function, we can get $u(x)$ must be
harmonic function.

Further we can extend the Brownian motion $\{B^x_t\}$ to the general
diffusion process $\{X^x_t\}$:

If $u(x)\in C_0^2(R^n)$ and satisfies
\begin{align}\label{X-harmonic}
\sum_{i}b_i\frac{\partial u}{\partial
x_i}(x)+\frac{1}{2}\sum_{i,j}(\sigma\sigma^\tau)_{i,j}\frac{\partial^2u}{\partial
x_i\partial x_j}(x)=0,
\end{align}
then we have $u(X^x_t)$ is a martingale for any x. The proof also
uses the It\^{o} formula. But conversely if $u(X^x_t)$ is a
martingale for any x, we can't conclude that $u(x)$ is smooth. Then
with additional assumption $u(x)\in C_0^2(R^n)$ we can get that
$u(x)$ satisfies the PDE(\ref{X-harmonic})(see \cite{oksendal1}).

Then naturally we will ask that what happens when we substitute the
expectation mechanism by the g-expectation mechanism. First we will
define the infinitesimal generator:
\begin{definition}
Let
\begin{align}
\mathcal{A}^X_gf(x):=\lim_{t\downarrow
0}\frac{\mathcal{E}^g_{0,t}[f(X^x_t)]-f(x)}{t},
\end{align}
then we call $\mathcal{A}^X_g$ the infinitesimal generator of a
diffusion process $\{X^x_t\}$ under g-expectations.
\end{definition}
Thanks to the celebrating nonlinear Feynman-Kac formula
\cite{Peng3}, we can get the explicit form of $\mathcal{A}^X_g$ when
$f \in C_0^2(R^n)$ by considering the following type of quasilinear
parabolic PDE:
\begin{equation}\label{quasilinear parabolic PDE}
\begin{cases}
 &\frac{\partial u}{\partial t}-\mathcal{L}u(t,x)-g(u(t,x),u_x(t,x)\sigma(x))=0,\\
 &u(0,x)=f(x).
\end{cases}
\end{equation}

Where
\begin{align}\label{mathcalL}
\mathcal{L}u(t,x)=\sum_{i}b_i\frac{\partial u}{\partial
x_i}(t,x)+\frac{1}{2}\sum_{i,j}(\sigma\sigma^\tau)_{i,j}\frac{\partial^2u}{\partial
x_i\partial x_j}(t,x).
\end{align}
When $f\in C_0^2(R^n)$, we assert that
\begin{align}\label{u(t,x)}
u(t,x)=\mathcal{E}^g_{0,t}[f(X^x_t)]
\end{align}
is the solution of PDE(\ref{quasilinear parabolic PDE}). Then under
the case $t=0$, we get
\begin{align}\label{smooth A}
\mathcal{A}^X_gf(x)=\mathcal{L}f(x)+g(f(x),f_x(x)\sigma(x)).
\end{align}
Then we finish the preliminary and we can introduce our main
results. In section 2, we give a characterization of g-harmonic
function under smooth case. In section 3, we characterize it under
continuous case, where the differential operator is interpreted as
viscosity solution. In section 4, we will investigate the strict
converse problem of mean value property of g-harmonic function
evoked by its classical counterpart \cite{oksendal}.

\section{Smooth Case}

The equality (\ref{smooth A}) has the implication about the relation
between the g-martingales and the g-harmonic functions when $f\in
C_0^2(R^n)$. In fact, the left side of (\ref{smooth A}) is related
to a g-martingale and the right side is related to a harmonic PDE.
At first we will give the definition of g-harmonic functions:

\begin{definition}
Let $f\in C_0^2(R^n)$. We call it a g-(super)harmonic function
w.r.t. $\{X^x_t\}$ if it satisfies
\begin{align}\label{g-harmonic def}
\mathcal{A}^X_gf(x)(\leq)=0\quad\quad \mbox{for any}\quad x\in R^n.
\end{align}
\end{definition}

Then we suffice to construct the relation between the
g-supermartingales and the g-superharmonic functions.

\begin{theorem}
If $f(x)\in C_0^2(R^n)$, then the following assertions are equivalent:\\
(1)$f(x)$ is a g-superharmonic function.\\
(2)$\{f(X^x_t)\}$ is a g-supermartingale for any $x\in R^n$.
\end{theorem}

\begin{proof}
(i) $(1)\Rightarrow (2)$:\\
$f\in C^2(R^n)$, by It\^{o} formula, we can get $f(X^x_t)$ is still
an It\^{o} diffusion process:
$$f(X^x_t)=f(X^x_s)+\int^t_s \mathcal{L}f(X^x_r)dr+\int^t_sf_x(X^x_r)\sigma(X^x_r)dB_r\quad 0\leq s\leq t.$$
and then we insert the term $g(f(X^x_r),f_x\sigma(X^r_x))$ and get:
\begin{align}
f(X^x_s)&=f(X^x_t)-\int^t_s
\mathcal{L}f(X^x_r)dr-\int^t_sf_x\sigma(X^x_r)dB_r\notag\\
&=f(X^x_t)+\int^t_sg(f(X^x_r),f_x\sigma(X^r_x))dr-\int^t_sf_x\sigma(X^x_r)dB_r\notag\\
&\quad-\int^t_s
\{\mathcal{L}f(X^x_r)+g(f(X^x_r),f_x\sigma(X^r_x))\}dr \notag
\end{align}
$f(x)$ is a g-superharmonic function, so:
$$\mathcal{L}f(X^x_r)+g(f(X^x_r),f_x\sigma(X^x_r))=\mathcal{A}^X_gf(X^x_r)\leq 0.$$
And then according to the comparison theory of
BSDE(cf.\cite{Peng2}), we can get $\{f(X^x_t)\}$ is a
g-supermartingale.

(ii) $(2)\Rightarrow(1)$:\\
By the definition of the $\mathcal{A}^X_g$:
$$\mathcal{A}^X_gf(x)=\lim_{t\downarrow 0}\frac{\mathcal{E}^g_{0,t}[f(X^x_t)]-f(x)}{t}.$$
$\{f(X^x_t)\}$ is a g-supermartingale, so:
$$\mathcal{E}^g_{0,t}[f(X^x_t)]-f(x)\leq 0,$$
then
$$\mathcal{A}^X_gf(x)\leq 0.$$
So we get $f(x)$ is a g-superharmonic function.
\end{proof}

\section{Continuous Case}
If we generalize the requirement of function $f(x)$ to be only
continuous on $R^n$, how we get a function $f$ which satisfies that
$f(X^x_t)$ is a g-martingale for any $x\in R^n$? With the help of
viscosity solution(cf. \cite{user guide}) we can also refer to the
quasi-linear second order PDEs. Here we need a lemma due to
Peng\cite{Peng3}.
\begin{lemma}\label{non Feynman Kac}
Let $0\leq t\leq T$ and
$$u(t,x)=\mathcal{E}^g_{0,T-t}[f(X^{x}_{T-t})].$$
Then $u(t,x)$ is the viscosity solution of the following PDE on
$(0,T)\times R^n$:
\begin{equation}\label{PDE**}
\begin{cases}
&\frac{\partial u}{\partial
t}+\mathcal{L}u(t,x)+g(u(t,x),u_x(t,x)\sigma (x))=0\\
&u(T,x)=f(x)
\end{cases}
\end{equation}
Here $g(y,z)$ and $f(x)$ satisfy:\\
(H2). Let $F(u,p)=g(u,p\sigma(x))$, then $\exists C>0$ s.t.
\begin{align}
&|F(u,p)|\leq C(1+|u|+|p|);\notag\\
&|D_uF(u,p)|, |D_pF(u,p)|\leq C,\notag
\end{align}
and (H3). $f(x)$ is a continuous function with a polynomial growth
at infinity.
\end{lemma}

\begin{definition}
Let $u(t,x) \in C(R \times R^n)$. $u(t,x)$ is said to be a viscosity
super-solution (resp. sub-solution) of the following
PDE(\ref{PDEa}):
\begin{align}\label{PDEa}
\frac{\partial u}{\partial
t}+\mathcal{L}u(t,x)+g(u(t,x),u_x(t,x)\sigma(x))=0,
\end{align}
if for any $(t,x) \in R \times R^n$ and $\varphi \in C^{1,2}(R
\times R^n)$ such that $\varphi(t,x)=u(t,x)$ and $(t,x)$ is a
maximum (resp. minimum)  point of $\varphi-u$,
\begin{align}
\frac{\partial \varphi}{\partial
t}(t,x)+\mathcal{L}\varphi(t,x)+g(\varphi(t,x),
\varphi_x(t,x)\sigma(x))\leq 0.\notag
\end{align}
\begin{align}
(\mbox{resp.}\quad\frac{\partial \varphi}{\partial
t}(t,x)+\mathcal{L}\varphi(t,x)+g(\varphi(t,x),
\varphi_x(t,x)\sigma(x))\geq 0.)\notag
\end{align}
$u(t,x)$ is said to be a viscosity solution of PDE(\ref{PDEa}) if it
is both a viscosity super- and sub-solution of (\ref{PDEa}).
\end{definition}

We also consider the viscosity solution of the following type of
quasilinear elliptic PDE(\ref{PDEb}):
\begin{align}\label{PDEb}
\mathcal{L}u(x)+g(u(x),u_x(x)\sigma(x))=0.
\end{align}

We can directly get an relation between the two solutions of
(\ref{PDEa}) and (\ref{PDEb}):

\begin{lemma}\label{pdea and pdeb}
Let $\tilde{u}(t,x)=u(x)$ for all $(t,x) \in R \times R^n$, then we have:\\
$\tilde{u}(t,x)$ is the viscosity super-(sub-)solution of PDE
(\ref{PDEa}) $\Leftrightarrow$ $u(x)$ is the viscosity
super-(sub-)solution of PDE (\ref{PDEb}).
\end{lemma}

\begin{proof}We only prove the case of viscosity super-solution. The "sub-" case is an
immediate conclusion of the "super-" case.

(i). $"\Rightarrow"$:

For any $(t_0,x_0)\in R\times R^n$, and a function $\varphi(x)\in
C^2(R^n)$ which satisfies $\varphi(x)\leq u(x),\varphi(x_0)=u(x_0)$,
we define $\tilde{\varphi}(t,x)=\varphi(x)$ for all $(t,x) \in R
\times R^n$. Then
$$\frac{\partial\tilde{\varphi}}{\partial t}=0,\ \ \  \tilde{\varphi}(t_0,x_0)=\tilde{u}(t_0,x_0),
\ \ \ \tilde{\varphi}(t,x)\leq\tilde{u}(t,x),$$ and due to the
assumption that $\tilde{u}(t,x)$ is the viscosity super-solution of
PDE(\ref{PDEa}), we have
\begin{align}
\frac{\partial \tilde{\varphi}}{\partial
t}(t_0,x_0)+\mathcal{L}\tilde{\varphi}(t_0,x_0)+g(\tilde{\varphi}(t_0,x_0),
\tilde{\varphi}_x(t_0,x_0)\sigma(x_0))\leq 0,\notag
\end{align}
i.e.
$$\mathcal{L}\varphi(x_0)+g(\varphi(x_0),
\varphi_x(x_0)\sigma(x_0))\leq 0.$$ So $u(x)$ is the viscosity
super-solution of PDE (\ref{PDEb}).

(ii). $"\Leftarrow"$:

For any $(t_0,x_0)\in R\times R^n$, and a function $\varphi(t,x)\in
C^2(R\times R^n)$ which satisfies $\varphi(t,x)\leq
\tilde{u}(t,x),\varphi(t_0,x_0)=\tilde{u}(t_0,x_0)$, then
\begin{align}\label{condition1}
\frac{\partial\varphi}{\partial t}(t_0,x_0)=0,
\end{align}
 and due to the assumption that $u(x)$ is the viscosity super-solution of
PDE(\ref{PDEb}), we have
$$\mathcal{L}\varphi(t_0,x_0)+g(\varphi(t_0,x_0),
\varphi_x(t_0,x_0)\sigma(x_0))\leq 0.$$
Combined with
(\ref{condition1}), we get
$$\frac{\partial\varphi}{\partial t}(t_0,x_0)+\mathcal{L}\varphi(t_0,x_0)+g(\varphi(t_0,x_0),
\varphi_x(t_0,x_0)\sigma(x_0))\leq 0.$$
 So $\tilde{u}(t,x)$ is the viscosity super-solution of PDE (\ref{PDEa}).
\end{proof}
Then we can introduce our main result of this section:
\begin{theorem}\label{theorem2}
We have the following two consequences:\\
(i). For any $f(x) \in C(R^n)$, and $g(y,z)$ satisfying (H1), if
$\forall x \in R^n$, $f(X^x_t)$ is a g-supermartingale, then $f(x)$
is
a viscosity super-solution of PDE (\ref{PDEb}).\\
(ii). For any $f(x)$ satisfying (H3), and $g(y,z)$ satisfying (H1)
(H2), let $f(x)$ is a viscosity super-solution of PDE (\ref{PDEb}),
then $\{f(X^x_t)\}$ is a g-supermartingale for all $x \in R^n$.
\end{theorem}

Actually, the consequence (ii) is the answer of our main problem and
the consequence (i) is the converse of it. But (i) is easier to be
proved, so we are going to prove (i) at first:
\begin{proof}(i).

For any $x \in R^n$, let $\varphi \in C^2(R^n)$, $\varphi(x) = f(x)$
where x is a maximum point of $\varphi-f$. It means $\forall
\tilde{x} \in R^n$, we have $\varphi(\tilde{x}) \leq f(\tilde{x})$.
Then from (\ref{smooth A}), we get
\begin{align}
\mathcal{L}\varphi(x)+g(\varphi(x), \varphi_x(x)\sigma(x)) &=
\mathcal{A}^X_g\varphi(x)\notag\\
&=\lim_{t\downarrow
0}\frac{\mathcal{E}^g_t[\varphi(X^x_t)]-\varphi(x)}{t}\notag\\
&=\lim_{t\downarrow
0}\frac{\mathcal{E}^g_t[\varphi(X^x_t)]-f(x)}{t}\notag.
\end{align}
According to the comparison theory of BSDE, we get
$$\mathcal{E}^g_t[\varphi(X^x_t)]
\leq \mathcal{E}^g_t[f(X^x_t)],$$
 and with the assumption $\{f(X^x_t)\}$ is a
g-supermartingale, we can get:
 $$\mathcal{E}^g_t[\varphi(X^x_t)]-f(x)
\leq \mathcal{E}^g_t[f(X^x_t)]-f(x)\leq0.$$
Then
\begin{align}
\mathcal{A}^X_g\varphi(x)=\lim_{t\downarrow
0}\frac{\mathcal{E}^g_t[\varphi(X^x_t)]-f(x)}{t} \leq 0.\notag
\end{align}
i.e.
$$\mathcal{L}\varphi(x)+g(\varphi(x), \varphi_x(x)\sigma(x))\leq 0.$$
By definition, it means $f(x)$ is a viscosity super-solution of
PDE (\ref{PDEb}).\\
(ii).

We want to prove $\{f(X^x_t)\}$ is a g-supermartingale for any $x\in
R^n$. It means that we need to prove $\forall x\in R^n\  and\
\forall 0\leq s\leq t$, we have
$$\mathcal{E}^g_{s,t}[f(X^x_t)]\leq f(X^x_s).$$
Under the assumption, in fact $b(x),\sigma(x)$ and $g(y,z)$ are all
independent of time t, so we can get the Markovian property of
$\mathcal{E}^g_{s,t}$, i.e.
$$\mathcal{E}^g_{s,t}[f(X^x_t)]=\mathcal{E}^g_{t-s}[f(X^y_{t-s})]|_{y=X^x_s}.$$
Then we get an equivalence relation:
\begin{align}\label{g-martingale equivalence}
&\{f(X^x_t)\} \mbox{is a g-(super)martingale for any}\  x \in R^n\Leftrightarrow\notag\\
& \mathcal{E}^g_{t}[f(X^x_{t})] =(\leq) f(x)\  \mbox{for any}\ t\geq
0\  \mbox{and}\  x \in R^n.
\end{align}
So we suffice to prove the latter assertion.

For any $T\geq 0$, the assumption $f(x)$ is a viscosity
super-solution of PDE(\ref{PDEb}) implies that
$\tilde{f}(t,x):=f(x)$ is a viscosity super-solution to the
following PDE:
\begin{equation}\label{PDE***}
\begin{cases}
&\frac{\partial u}{\partial
t}+\mathcal{L}u(t,x)+g(u(t,x),u_x(t,x)\sigma (x))=0,\\
&u(T,x)=f(x),
\end{cases}
\end{equation}
according to lemma \ref{pdea and pdeb}. And with the help of lemma
\ref{non Feynman Kac},
$$u(t,x)=\mathcal{E}^g_{0,T-t}[f(X^{x}_{T-t})]$$
is actually the viscosity solution of PDE (\ref{PDE***}). Moreover
by the maximum principle of the viscosity solution( see
\cite{Barles}), we can get:
$$u(t,x) \leq \tilde{f}(t,x)\quad \mbox{for any}\ 0\leq t\leq T.$$
Especially, we have
$$u(0,x)\leq \tilde{f}(0,x),$$
i.e.
$$\mathcal{E}^g_T[f(X^x_T)]\leq f(x).$$
\end{proof}
\begin{corollary}
(i). For any $f(x) \in C(R^n)$, and $g(y,z)$ satisfying (H1), if
$\forall x \in R^n$, $f(X^x_t)$ is a g-martingale, then $f(x)$ is
a viscosity solution of PDE(\ref{PDEb}).\\
(ii). For any $f(x)$ satisfying (H3), and $g(y,z)$ satisfying (H1)
(H2), let $f(x)$ is a viscosity solution of PDE(\ref{PDEb}), then
$\{f(X^x_t)\}$ is a g-martingale for all $x \in R^n$.
\end{corollary}

It is an immediate consequence from the theorem \ref{theorem2}.

\section{Strict Converse of Mean Value Property}

For classical harmonic function, many generalized results of the
converse problem of mean value property have been investigated (cf.
\cite{Kellogg}\cite{oksendal}). In \cite{oksendal}, {\O}ksendal and
Stroock give a technique to solve a strict converse of the mean
value property for harmonic functions. Now we will generalize it to
the case of g-harmonic function. Here the strictness means that for
each $x\in R^n$ we don't need justify that for any stopping time
$\tau$ whether $\mathcal{E}_{0,\tau}^g(f(X^x_{\tau}))$ equals
$f(x)$. We only need to justify one appropriate stopping time of
each x.

In the sequel we put $\Delta(x,r)=\{y\in R^n; |y-x|<r\}$ for any
$x\in R^n$ and $r>0$. Let $\tau_{U}=\mbox{inf}\{t>0; X^x_t \in
U^c\}$ for any open set U. And we suppose the operator
(\ref{mathcalL}) is elliptic on $R^n$.
\begin{theorem}\label{theorem3}
$f(x)$ is a local bounded continuous function on $R^n$. If for any
$x\in R^n$, there exists a radius $r(x)$, the mean value property
holds:
\begin{align}\label{rx}
\mathcal{E}_{0,\tau_x}^g[f(X^x_{\tau_x})]=f(x)\ \ \ \ \ \mbox{here}\
\ \  \tau_x=\tau_{\Delta(x,r(x))}.
\end{align}
And $r(x)$ is a measurable function of x and satisfies that for each
x, there exists a bounded open set $U_x$, $x\in U_x$ and moreover
$r(y),\ y\in U_x$ should satisfy the following two conditions:
\begin{align}\label{rycondition1}
0\leq r(y)\leq \mbox{dist}(y,\partial U_x),
\end{align}
and
\begin{align}\label{rycondition2}
\mbox{inf}\{r(y); y\in K\} > 0
\end{align}
for all closed subsets K of $U_x$ with $\mbox{dist}(K,\partial U_x)
> 0$. Then we can get:

(i). For each $y\in U_x$ the mean value property holds on the
boundary:
\begin{align}
\mathcal{E}^g_{0,\tau_y}[f(X^y_{\tau_y})]=f(y),\ \ \ \mbox{here}\ \
\tau_y =\mbox{inf}\{t>0; X^y_t \in U_x^c\}. \notag
\end{align}
and furthormore:

(ii). $f(x)$ is the viscosity solution of PDE(12).
\end{theorem}

\begin{proof}
$(i)\Rightarrow (ii)$ is also based on the nonlinear Feynman-Kac
formula for elliptic PDE(cf.\cite{Peng3}). So we sufficiently prove
the first conclusion.

For each $y\in U_x$, we define a sequence of stopping times $\tau_k$
for $\{X_t^y\}$ by induction as follows:
\begin{align}
\tau_0 &\equiv 0\notag \\
\tau_k &= \mbox{inf} \{t\geq \tau_{k-1};
|X^y_t-X^y_{\tau_{k-1}}|\geq r(X^y_{\tau_{k-1}})\}; \ \ k\geq
1.\notag
\end{align}
By the mean property(\ref{rx}), and the strong markovian property we
can get
\begin{align}
\mathcal{E}^g_{0,\tau_k}[f(X^y_{\tau_k})]&=\mathcal{E}^g_{0,\tau_{k-1}}[\mathcal{E}^g_{\tau_{k-1},\tau_k}
[f(X^y_{\tau_k})]]\notag \\
&=\mathcal{E}^g_{0,\tau_{k-1}}[\mathcal{E}^g_{0,\tau_k-\tau_{k-1}}[f(X^{X^y_{\tau_{k-1}}}_{\tau_k-\tau_{k-1}})]]\notag\\
&=\mathcal{E}^g_{0,\tau_{k-1}}[f(X^y_{\tau_{k-1}})],\notag
\end{align}
then by induction we get
\begin{align}\label{induction}
\mathcal{E}^g_{0,\tau_k}[f(X^y_{\tau_k})]=f(y).\notag
\end{align}
In the following we will prove $\tau_k\rightarrow \tau_y\ a.e.$ when
$k \rightarrow \infty$. Obviously
$$\tau_k \geq \tau_{k-1},$$
so there exists a stopping time $\tau$ s.t. $\tau_k\uparrow \tau$.
If $\tau \neq \tau_y$, then there exists $\epsilon >0$ s.t.
$$\mbox{dist}(X^y_{\tau_k},\partial U_x)\geq \epsilon\ \ \ \mbox{for any}\ k.$$
Let $r_k=r(X^y_{\tau_k})$, according to the
condition(\ref{rycondition2}), we get there exists $r>0$,
$$r_k \geq r\ \ \ \ \  \mbox{for any}\ \  k.$$
It means
$$\mbox{dist}(X^y_{\tau_k},X^y_{\tau_{k-1}})\geq r.$$
And since ${X^y_t}$ is continuous, then $\tau_k\rightarrow\infty$,
which implies $\tau_y=\infty$. So
$$P(\tau_k\  \mbox{don't converge to}\  \tau_y)\leq P(\tau_y=\infty).$$
But for (\ref{mathcalL}) is elliptic and $U_x$ is bounded, we have
$P(\tau_y<\infty)=1$. So
$$P(\tau_k\ \mbox{converge to}\ \tau_y)=1.$$
Then we get
\begin{align}
f(y)&=\mathcal{E}^g_{0,\tau_k}[f(X^y_{\tau_k})]\notag\\
&=\lim_{k\rightarrow
\infty}\mathcal{E}^g_{0,\tau_k}[f(X^y_{\tau_k})]\notag\\
&=\mathcal{E}^g_{0,\tau_y}[f(X^y_{\tau_y})]\notag
\end{align}
So we have finished the proof.
\end{proof}

\textbf{Acknowledgement.} The author is grateful to S. Peng for the
suggestion to use maximal principle of viscosity solution to prove
theorem 2(ii). This work was supported by the National Natural
Science Foundation of China (No. 11026125).


\end{document}